\documentclass[12pt,reqno]{amsart}

\usepackage{amsmath, amsfonts, amsthm, amssymb, color,  graphicx, mathrsfs, cite}
\usepackage{stmaryrd}

\usepackage{color}

\theoremstyle{plain}
\newtheorem{Theorem}{Theorem}[section]
\newtheorem{Lemma}{Lemma}[section]
\newtheorem{Proposition}{Proposition}[section]

\theoremstyle{Definition}

\theoremstyle{Remark}
\newtheorem{Remark}{Remark}[section]

\numberwithin{equation}{section}
\allowdisplaybreaks

\usepackage{ifpdf}
\ifpdf \usepackage[colorlinks=true, citecolor=blue, linkcolor=blue, urlcolor=blue]{hyperref} \fi

\begin{document}

\title[ Large Time Behavior  of
Compressible Fluid without Viscosity]
{ Asymptotic stability of viscous contact wave and rarefaction waves for the system of
heat-conductive ideal gas without viscosity}

\author{ Lili Fan}
\address{Lili Fan: School of Mathematics and Computer Science, Wuhan Polytechnic University, Wuhan 430023, P. R. China}
\email{fll810@live.cn}

\author{ Guiqiong Gong}
\address{ Guiqiong Gong:
School of Mathematics and Statistics, Wuhan University, Wuhan 430072, P. R. China}
\email{gongguiqiong@yeah.net}

\author{Shaojun Tang*}\thanks{*Corresponding author.
Email address: shaojun.tang@whu.edu.cn~(S. J. Tang)}
\address{Shaojun Tang:
 School of Mathematics and Statistics, Wuhan University, Wuhan 430072, P. R. China}
\email{shaojun.tang@whu.edu.cn}

\date{} \maketitle

This paper is concerned with the Cauchy problem of heat-conductive ideal gas without viscosity. We show that, for the non-viscous case, if the strengths of the wave patterns and the initial perturbation are suitably small, the unique global-in-time solution exists and asymptotically tends toward the corresponding the viscous contact wave or the composition of a viscous contact wave with rarefaction waves determined by the initial condition, which extended the results by Huang-Li-Matsumura \cite{Huang-Li-Matsumura-ARMA-2010}, where they treated the viscous and heat-conductive ideal gas.\\



\tableofcontents

\section{Introduction}
\setcounter{equation}{0}

 We consider the Cauchy problem for the equation of heat-conductive ideal gas without viscosity
\begin{eqnarray}\label{1.1}
\left\{
\begin{array}{ll}
  v_t - u_x = 0, \\[2mm]
  u_t + p_x = 0, \\[2mm]
  ( e + \frac{u^{2}}{2} )_{t} + ( p u )_x = \kappa \big( \frac{\theta_x}{v} \big)_x\, , \quad \quad x \in {\mathbb{R}}, \ t>0\,
\end{array}
\right.
\end{eqnarray}
with the following prescribed initial data and the far field state
\begin{eqnarray}\label{1.3}
\left\{
\begin{array}{l}
  (v, u, \theta)(x,0) = (v_0, u_0, \theta_0)(x), \quad x \in {\mathbb{R}} \,, \\[2mm]
  (v, u, \theta)(\pm \infty, t) = (v_\pm, u_\pm, \theta_\pm), \quad t > 0 \,.
\end{array}
\right.
\end{eqnarray}
Here $v(x,t)>0$, $u(x,t)$, $\theta(x,t)>0$, $e(x,t)>0$ and $p(x,t)$ are the specific volume, fluid velocity, internal energy, absolute temperature and pressure of the gas, respectively, while $\kappa>0$ denotes the heat conduction coefficient. $v_\pm (>0)$, $\theta_\pm (>0)$ and $u_\pm (\in{\mathbb{R}})$ are given constants and the initial data $(v_0(x), u_0(x),\theta_0(x))$ are assumed to satisfy $\inf \limits_{x \in {\mathbb{R}} } v_0(x) > 0$, $\inf \limits_{x \in {\mathbb{R}} } \theta_0(x) > 0$ and the compatibility conditions $(v_0, u_0, \theta_0) (\pm \infty) = (v_\pm, u_\pm, \theta_\pm)$.

Throughout this paper, we are concerned with the ideal and polytropic fluids and in such a case, $p$ and $e$ are given by the following state equations
\begin{equation}\label{1.2}
  p = \frac{R\theta}{v} = A v^{-\gamma} e^{\frac{\gamma-1}{R} s} \, ,\quad
  e = \frac{R}{\gamma-1} \theta + const. \, ,
\end{equation}
where $s$ is the entropy, $\gamma>1$ is the adiabatic exponent and both $A$ and $R$ are  positive constants.

If $\kappa=0$, we can rewrite the system (\ref{1.1}) as
\begin{eqnarray}\label{1.4}
\left\{
\begin{array}{ll}
  v_t - u_x = 0, \\[2mm]
  u_t + p_x = 0, \\[2mm]
  ( e + \frac{u^{2}}{2} )_{t} + ( p u )_x =0.
\end{array}
\right.
\end{eqnarray}
It is well-known that the above system has the three eigenvalues:
$\lambda_1=-\sqrt{\gamma p/v}<0,\ \lambda_2=0,\ \lambda_3=-\lambda_1>0$,
where the second characteristic field is linear degenerate and the others are genuinely nonlinear. The solutions  of this inviscous equations with the Riemann initial data
\begin{eqnarray}\label{1.5}
  (v, u, \theta)(x,0)=
\left\{
\begin{array}{l}
  (v_-, u_-, \theta_-), x<0\, , \\[2mm]
  (v_+, u_+, \theta_+), x>0\, .
\end{array}
\right.
\end{eqnarray}
has the basic Riemann solutions which are dilation invariant: shock wave, rarefaction wave, contact discontinuity, and certain linear combinations of these basic wave patterns \cite{Smoller-Sringer-1994}. The inviscid system (\ref{1.4}) is a typical example of hyperbolic conservation laws and is of great importance to study the large-time asymptotic behavior of solutions of the corresponding viscous system (\ref{1.1}). In fact, if the unique global entropy to its Riemann problem \eqref{1.4}, \eqref{1.5} consists of shock wave, rarefaction wave, contact discontinuity, and/or their linear superpositions, then the large time behavior of the Cauchy problem \eqref{1.1}, \eqref{1.3} of the corresponding viscous conservation laws is expected to be precisely described by the corresponding viscous shock wave, rarefaction wave, viscous contact discontinuity and/or their linear superpositions.

The rigorous mathematical justifications of the above expectation is one of the hottest topics in the field of nonlinear partial differential equations, especially in the field of nonlinear dissipative hyperbolic conservation laws. In fact, new phenomena have been discovered and new techniques, such as the weighted characteristic energy method and the approximate Green function method have been developed based on the intrinsic properties of the underlying nonlinear waves, cf. \cite{Huang-Li-Matsumura-ARMA-2010, Liu-MAMA-1985, Liu-CPAM-1986, Liu-Zeng-MAMS-1997, Liu-Zeng-CMP-2009, Szepessy-Xin-ARMA-1993, XinZhouping-WS-1996} and the references cited therein.

By combining these methods together with the fundamental energy method, some interesting results have been obtained for the compressible Navier-Stokes equations:

\begin{itemize}
  \item For the case when both viscosity coefficient $\mu$ and the heat conductivity coefficient $\kappa$ are positive constants, many excellent results have been obtained for the nonlinear stability of some basic wave patterns consisting of diffusion waves, viscous shock waves, rarefaction waves, viscous contact discontinuities and/or their certain linear superpositions with small perturbation, cf. \cite{Liu-Zeng-MAMS-1997} for diffusion waves,  \cite{Liu-CPAM-1986, Liu-MAMA-1985, Kawashima-Matsumura-CMP-1985, Matsumura-Nishihara-JJAM-1985} for viscous shock profiles, \cite{Liu-Xin-CMP-1988, Matsumura-Nishihara-JJAM-1985} for rarefaction waves, \cite{Huang-Matsumura-Shi-OJM-2004, Huang-Matsumura-Xin-ARMA-2005, Huang-Xin-Yang-AM-2008} for viscous contact discontinuities, \cite{Huang-Matsumura-CMP-2009} for the composition of viscous shock profiles of different family, \cite{Huang-Li-Matsumura-ARMA-2010} for the  composition a viscous contact wave and rarefaction waves. For the corresponding results with large initial perturbation, see \cite{Duan-Liu-Zhao-TransAMS-2009, F-L-W-Zhao-JDE-2014, Huang-Wang-IUMJ-2016, Huang-Zhao-RSMUP-2003, Matsumura-Nishihara-CMP-1992, Matsumura-Nishihara-QAM-2000, Nishihara-Yang-Zhao_SIAMJMA_2004, Wan-Wang-Zhao-JDE-2016, Wang-Zhao-Zou-KRM-2013} and the reference cited therein.

  \item For the case when $\mu > 0$, $\kappa = 0$, Liu and Zeng \cite{Liu-Zeng-JDE-1999} studied the large-time behavior of solutions around a constant state for this case. For the three-dimensional case, the global existence and the temporal  convergence rate of solutions was obtained by Duan and Ma \cite{Duan-Ma-Indiana-2008}.
  \item When $\mu = 0$, $\kappa >0$, if the corresponding Riemann solution consists of two shock waves form different families, the nonlinear stability of the composition of viscous shock waves from different families was investigated by Fan and Matsumura in \cite{Fan-Matsumura-JDE-2015}, while the nonlinear stability of viscous contact waves together with the temporal decay rates was obtained by Ma and Wang in \cite{Ma-Wang-JMP-2016}. In all these results, both the strengths of the underlying wave patterns and the initial perturbation are assumed to be small.
\end{itemize}

From the above results, a natural question is: \emph{Whether can we get the nonlinear stability of combination of viscous contact wave with rarefaction waves for the case $\mu=0$ and $\kappa>0 ?$} This is the motivation of our work. As a continuation of \cite{Fan-Matsumura-JDE-2015}, this manuscript showed that, if the strengths of the viscous waves and the initial perturbation are suitably small, the unique global-in-time solution exists and asymptotically tends toward the corresponding viscous contact wave or the composition of a viscous contact wave with rarefaction waves from different families. Our result generalizes the corresponding results obtained by Huang, Li and Matsumura in \cite{Huang-Li-Matsumura-ARMA-2010} for the case of $\mu>0, \kappa>0$ to the case of $\mu>0, \kappa=0$ and extends the result of Ma and Wang \cite{Ma-Wang-JMP-2016} for the nonlinear stability of the viscous contact waves to the nonlinear stability of the composition of a viscous contact wave with rarefaction waves from different families.

Now we outline the main difficulties of the problem and our strategy to overcome the difficulties. The first difficulty is due to the fact that the system \eqref{1.1} is of less dissipation caused by the fact that $\mu=0$, thus we need more subtle estimates to recover the regularity and dissipation for the components of the hyperbolic part. We shall overcome this difficulty by manipulating several new energy estimates and also looking for the perturbed solution for the integrated system of (\ref{1.1}) in $C([0, \infty), H^2)$ to control the nonlinearity of the hyperbolic part, instead of the usual $C([0, \infty), H^1)$ in \cite{Huang-Li-Matsumura-ARMA-2010}. Secondly, our stability results included the superposition of the rarefaction waves with the viscous contact discontinuity, which will lead to the degenerate characteristics. Thanks to the  estimates on the heat kernel function obtained by Huang, Li and Matsumura established in \cite{Huang-Li-Matsumura-ARMA-2010}, we can still close the energy type estimates for our non-viscous case and get the desired results.

Before concluding this section, it is worth pointing out that there are many results on the nonlinear stability of basic wave patterns to some hyperbolic conservation laws with dissipation,  cf. \cite{Goodman_ARMA-1986, Liu-Zeng-CMP-2009, Szepessy-Xin-ARMA-1993} for hyperbolic conservation laws with artificial viscosity, \cite{He-Tang-Wang-AMSSB-2016, Liu-Yang-Zhao-Zou-SIAMJMA-2014, Huang-Liao-M3AS-2017, Huang-Wang-Xiao-KRM-2016, Wan-Wang-JDE-2017, Wang-Zhao-M3AS-2016} for compressible Navier-Stokes equations with density and/or temperature dependent transportation coefficients, and \cite{Chen-He-Zhao-JMAA-2015, Chen-Xiao-MMAS-2013, Chen-Xiong-Meng-JMAA-2014, Chen-Zhao-JMPA-2014} for compressible fluid models of Korteweg type, and so on.

The rest of the paper is arranged in the following way: in the next section, we will give some elementary properties of the viscous  contact wave  and rarefaction wave and state the main results. Main theorem will be proved in the section 3.

\section{Preliminaries and Main Results}
\setcounter{equation}{0}
To show our main results,
in this section, we will construct the two desired viscous contact wave and viscous rarefaction waves
for (\ref{1.1}) and state the main results.  For each
$z_-:=(v_{-},u_{-},\theta_{-})$,  we can see our
situation takes place provided $z_+:=(v_+,u_+,\theta_+)$
is located on a quarter of a curved surface in a small neighborhood of $z_-$.
In what follows, as \cite{Fan-Matsumura-JDE-2015},  the neighborhood of $z_-$ denoted by
$\Omega_-$ is given by
$$
 \Omega_- = \{(v, u, \theta) |\ | ( v - v_-, u-u_-, \theta - \theta_- ) | \le \bar{\delta} \},
$$
where $\bar{\delta}$ is a positive constant depending only on $z_-$.

\subsection{Viscous Contact Wave}
we firstly recall the viscous contact wave $(\widetilde{v},\widetilde{u},\widetilde{\theta})$  for the compressible system
(\ref{1.1}) defined in \cite{Huang-Xin-Yang-AM-2008}. For the Riemann problem (\ref{1.4}),(\ref{1.5}),
it is known that the contact discontinuity solution $\widetilde{Z}(x,t):=(\widetilde{V},\widetilde{U},\widetilde{\Theta})(x,t)$ takes the form
\begin{eqnarray}\label{2.1}
(\widetilde{V},\widetilde{U},\widetilde{\Theta})(x,t)=
\left\{
\begin{array}{l}
  (v_-, u_-, \theta_-), \ x<0,\ t>0 \, .\\[2mm]
  (v_+, u_+, \theta_+), \ x>0,\ t>0 \, .
\end{array}
\right.
\end{eqnarray}
provided that
\begin{eqnarray}\label{2.2}
  u_- = u_+, \quad p_- = \frac{R \theta_-}{v_-} = \frac{R \theta_+}{v_+} = p_+\, .
\end{eqnarray}
In the setting of the compressible Navier-Stokes system (\ref{1.1}),
the smooth approximate wave $(\widetilde{v},\widetilde{u},\widetilde{\theta})$ to the contact wave
 behaves as a diffusion wave
due to the dissipation effect and  we call this wave "viscous contact wave".
it can be constructed as follows.
Since the pressure  is constant under the condition
(\ref{2.7}), we set
\begin{eqnarray}\label{2.3}
  p_+ = \frac{R \widetilde{\theta}}{\widetilde{v}} \, ,
\end{eqnarray}
which indicates the leading part of the energy equation $(\ref{1.1})_3$ is
\begin{eqnarray}\label{2.4}
  \tfrac{R}{\gamma-1} \widetilde{\theta}_t + p_+ \widetilde{u}_x = \kappa \Big( \tfrac{\widetilde{\theta}_x}{\widetilde{v}} \Big)_x \, .
\end{eqnarray}
Meanwhile the equation $\widetilde{v}_t=\widetilde{u}_x$ leads to a nonlinear diffusion equation
\begin{eqnarray}\label{2.5}
\left\{
\begin{array}{l}
  \widetilde{\theta}_t = a \Big( \frac{\widetilde{\theta}_x}{\widetilde{\theta}} \Big)_x, \quad a= \tfrac{\kappa p_+ (\gamma-1)}{\gamma R^2}, \\[2mm]
  \Theta(\pm,t) = \theta_{\pm} \, .
\end{array}
\right.
\end{eqnarray}
which has a unique self similarity solution $\widetilde{\theta}(x,t)=\widetilde{\theta}(\xi),\ \xi=\frac{x}{\sqrt{1+t}}$. Furthermore, on the one hand, $\widetilde{\theta}(\xi)$ is a monotone function, increasing if $\theta_+>\theta_-$ and decreasing if $\theta_+<\theta_-$; on the other hand, there exists some positive constant $\delta_0$, such that for $\delta = |\theta_+-\theta_-| \leq \delta_0 (\leq \overline{\delta})$, $\widetilde{\theta}$ satisfies
\begin{eqnarray}\label{2.6}
  (1+t) | \widetilde{\theta}_{xx} | + (1+t)^{\frac{1}{2}} | \widetilde{\theta}_{x} | +
  | \widetilde{\theta} - \theta_\pm | \lesssim \delta e^{-\frac{x^2}{1+t}},   \quad as \quad
  | x | \rightarrow \infty \, .
\end{eqnarray}
Once $\widetilde{\theta}$ is defined, the viscous contact profile
$$Z^c(x,t):=(V^c,U^c,\Theta^c)(x,t)$$
 is  determined as follows:
\begin{eqnarray}\label{2.7}
  \Theta^c = \widetilde{\theta}, \quad V^c = \frac{R}{p_+} \widetilde{\theta} \, , \quad
  U^c = u_- + \frac{\kappa (\gamma-1)}{\gamma R} \frac{\widetilde{\theta}_{x}}{\widetilde{\theta}}
  \, .
\end{eqnarray}
It is straightforward to check that  $Z^c(x,t)$ satisfies
\begin{eqnarray*}
  \| (Z^c - \widetilde{Z})(t) \|_{L^p} = O(\kappa^\frac{1}{2p}) (1+t)^\frac{1}{2p},\ p \geq 1 \, ,
\end{eqnarray*}
which means the nonlinear diffusion wave $Z^c(x,t)$ approximates the contact discontinuity
$\widetilde{Z}(x,t)$ to the Euler system (\ref{1.4}) in $L^p(p\geq1)$
norm. Moreover, the viscous contact wave  $Z^c(x,t)$
solves the compressible Navier-Stokes system  without viscosity (\ref{1.1}) as
\begin{eqnarray}\label{2.8}
\left\{
\begin{array}{ll}
  V^c_t - U^c_x = 0 \, ,\\[2mm]
  U^c_t + \Big( \frac{R\Theta^c}{V^c} \Big)_x = U^c_t \, , \\[2mm]
  \tfrac{R}{\gamma-1} \Theta^c_t + p_+ U^c_x = \kappa \Big( \frac{\Theta^c_x}{V^c} \Big)_x \, .
\end{array}
\right.
\end{eqnarray}

Our first main result is as follow:
\begin{Theorem}
For any given $z_-$, assume that $z_+\in \Omega_-$ satisfies (\ref{2.2}),
let $Z^c(x,t)$ is the viscous contact wave defined in (\ref{2.7}) with strength
$\delta=|\theta_+-\theta_-|$. There exist positive constants $\epsilon_1$
and $\delta_1(\leq \delta_0)$, such that if $\delta<\delta_1$ and the initial data satisfies
\begin{eqnarray}\label{2.9}
  \| (v_0(\cdot) - V(\cdot,0), u_0(\cdot) - U(\cdot,0), \theta_0(\cdot) - \Theta(\cdot,0)) \|_2 \leq \epsilon_1,
\end{eqnarray}
then the Cauchy problem (\ref{1.1}), (\ref{1.2}) admits a unique global solution $(v,u,\theta)(t,x)$ satisfies
\begin{eqnarray*}
\begin{aligned}
  ( v-V^c, u-U^c, \theta-\Theta^c )(t,x) \in X[0,+\infty)
\end{aligned}
\end{eqnarray*}
and
\begin{eqnarray}\label{2.10}
  \lim \limits_{t \rightarrow \infty} \sup \limits_{x \in {\mathbb{R}}} |( v-V^c, u-U^c, \theta-\Theta^c )(x,t)| = 0 \, ,
\end{eqnarray}
where the solution space $X(0,t)$ is defined in (\ref{3.5}).
\end{Theorem}

\subsection{Composition Waves}

When the relation (\ref{2.2}) is fails, the basic theory  of conservation laws \cite{Smoller-Sringer-1994} implies that for any given constant state $z_-$ and $z_+ \in \Omega_-$ ( $\overline{\delta}$ is suitably small ), the Riemann problem (\ref{1.4}), (\ref{1.5}) has  a unique solution.
Based on this, our second purpose is concerned with the stability of superposition of a viscous contact wave with rarefaction
waves. In this situation, we assume that
\begin{eqnarray}\label{2.11}
  z_+ \in R_1CR_3(z_-)\subseteq \Omega_- \, ,
\end{eqnarray}
where
\begin{eqnarray}\label{2.12}
  && R_1CR_3(z_-):= \bigg\{ (v,u,\theta) \in \Omega_- \bigg| s \neq s_- \\[2mm]
  && u\geq u_- - \int^{e^{\frac{\gamma-1}{R\gamma} ( s-s_- ) v}}_{v_-} \lambda_1 (\eta,s_-) d \eta, \quad u\geq u_- - \int^{v_-}_{e^{\frac{\gamma-1}{R\gamma}(s-s_-)v}} \lambda_3 (\eta,s) d \eta \bigg\} \nonumber
\end{eqnarray}
with the entropy $ s $ in (\ref{1.2}) is defined as follows:
\begin{eqnarray}\label{2.13}
  s = \frac{R}{\gamma-1} \ln \frac{R \theta}{A} + R \ln v \, , \quad
  s_\pm = \frac{R}{\gamma-1} \ln \frac{R \theta_{\pm}}{A} + R \ln v_\pm \, .
\end{eqnarray}
It is known that there exists some suitably small $\delta_1>0$ such that for
$$|\theta_--\theta_+|\leq \delta_1,$$
there exists  a unique pair of points $z^m_-:=(v^m_-,u^m,\theta^m_-)$
and $z^m_+:=(v^m_+,u^m,\theta^m_+)$ in $\Omega_-$ satisfying
\begin{eqnarray}\label{2.14}
  \frac{R \theta^m_-}{v^m_-}=\frac{R \theta^m_+}{v^m_+}:=p_m
\end{eqnarray}
 and
 \begin{eqnarray}\label{2.15}
|v^m_\pm-v_\pm|+|u^m-u_\pm|+|\theta^m_\pm-\theta_\pm| \lesssim |\theta_+-\theta_-| \, .
\end{eqnarray}
 Moreover, the point $z^m_-$ belongs to the 1-rarefaction wave curve
 $z^r_1:=(v^r_1,u^r_1,\theta^r_1)(\frac{x}{t}) $  which connected with  $z_-$,
 while $z^m_+$ belongs to the 3-rarefaction wave curve
 $z^r_3:=(v^r_3,u^r_3,\theta^r_3)(\frac{x}{t}) $ which connected with  $z_+$. that is, the 1-rarefaction wave is the weak solution of Riemann problem of the Euler system (\ref{1.4}) with the following Riemann data
 \begin{eqnarray*}
  z^r_1(x,0)=
\left\{
\begin{array}{l}
  (v_-,u_-,\theta_-),\ x<0,\\[2mm]
  (v^m_-,u^m,\theta^m_-),\ x>0
\end{array}
\right.
\end{eqnarray*}
 and the 3-rarefaction wave with Riemann data as
 \begin{eqnarray*}
  z^r_3(x,0)=
\left\{
\begin{array}{l}
  (v^m_+,u^m,\theta^m_+),\ x<0 \, ,\\[2mm]
  (v_+,u_+,\theta_+),\ x>0 \, .
\end{array}
\right.
\end{eqnarray*}
Since rarefaction waves $z^r_i(\frac{x}{t})(i=1,3)$  are not smooth enough, it is convenient to construct smooth approximations of them. Motivated by \cite{Kawashima-Matsumura-CMP-1985}, we start with the problem of the  Burgers equation:
\begin{eqnarray}\label{2.16}
\left\{
\begin{array}{l}
  w^r_t + w^r w^r_x = 0, \quad x \in {\mathbb R},\ t>0,\\[2mm]
  w^r(0,x) = w^r_0(x):= \tfrac{1}{2} \big( w_r + w_l \big) +
\tfrac{1}{2} \big( w_r - w_l \big) \tanh(x),
\end{array}
\right.
\end{eqnarray}
where $w_l = \lambda_1 (v_-, s_-)$, $w_r = \lambda_1 (v_+, s_-)$. Let $w(x,t)$ be the unique global solution of (\ref{2.16}), then $Z^r_1(x,t) := (V^r_1, U^r_1, \Theta^r_1) (x,t)$ are defined by
\begin{eqnarray}\label{2.17}
\left\{
\begin{array}{ll}
  \lambda_1 (V^r_1,s_-) = w(x,t),\\[2mm]
  U^r_1 = u_- - \int^{V^r_1}_{v_{-}} \lambda_1 (\eta, s_-) d \eta,\\[2mm]
  \Theta^r_1 = \theta_- (v_-)^{\gamma - 1} (V^r_1)^{1 - \gamma}
\end{array}
\right.
\end{eqnarray}
with
$$w_l = \lambda_1 (v_-,s_-), \quad w_r = \lambda_1 (v^m_-, s_-)$$
and the smooth approximation of  $(v^r_3, u^r_3, \theta^r_3)(\frac{x}{t})$ is given by
$Z^r_3(x,t) := (V^r_3, U^r_3, \Theta^r_3)(x,t)$  constructed by the same manner as (\ref{2.17})
\begin{eqnarray}\label{2.18}
\left\{
\begin{array}{ll}
  \lambda_3 (V^r_3, s_+) = w(x,t),\\[2mm]
  U^r_3 = u_+ - \int^{V^r_3}_{v_{+}} \lambda_3 (\eta, s_+) d \eta,\\[2mm]
  \Theta^r_3 = \theta_+ (v_+)^{\gamma-1} (V^r_3)^{1-\gamma},
\end{array}
\right.
\end{eqnarray}
where $w(x,t)$ is defined in (\ref{2.16}) with
$$w_l = \lambda_3(v^m_+, s_+), \quad w_r = \lambda_3 (v_+, s_+).$$
And since the condition (\ref{2.14}), (\ref{2.15}), $z^m_\pm$ is connected by
the viscous contact wave $Z^{c}(x,t)$ constructed in
 (\ref{2.5}), (\ref{2.7}). Finally, we investigate some properties of
$Z^r_i(x,t)(i=1,3)$  and $Z^{c}(x,t)$ as \cite{Huang-Li-Matsumura-ARMA-2010},
divided ${\mathbb{R}}\times [0,t]$ into three parts
${\mathbb{R}} \times [0,t] = \Omega_1 \cup \Omega_c \cup \Omega_3$ with
\begin{eqnarray}\label{2.19}
\begin{aligned}
  &\Omega_1=\big\{(x,t)|2x<\lambda_1(v_-^m,s_-)t\big\},\\[2mm]
  &\Omega_3=\big\{(x,t)|2x>\lambda_3(v_+^m,s_+)t\big\},\\[2mm]
  &\Omega_c=\big\{(x,t)|\lambda_1(v_-^m,s_-)t\leq 2x \leq\lambda_3(v_+^m,s_+)t\big\},
\end{aligned}
\end{eqnarray}
it holds that
\begin{Lemma}(\cite{Huang-Li-Matsumura-ARMA-2010})
For any given $z_-$, assume that $z_+ \in R_1CR_3(z_-)\subset \Omega_-$,
then the smooth rarefaction wave $Z^r_i(x,t)(i=1,3)$
and $Z^c$ satisfy:\\
$(i)$ \ $V^r_{it} = U^r_{ix} > 0(i=1,3)$ for all $x \in {\mathbb{R}},\ t>0$.\\
$(ii)$ \ For $1\leq p \leq \infty$ and $i=1,3$, it holds
\begin{eqnarray*}
  &&\| (V^r_{ix}, U^r_{ix}, \Theta^r_{ix}) (t) \|_{L^p} \lesssim \min \{ \delta, \delta^\frac{1}{p} (1+t)^{-1 + \frac{1}{p}} \},\\[2mm]
  &&\| (V^r_{ixx}, U^r_{ixx}, \Theta^r_{ixx}) (t) \|_{L^p} \lesssim \min \{ \delta, (1+t)^{-1} \}.
\end{eqnarray*}
$(iii)$ \ In $\Omega_c$, we have for $i=1,3$
\begin{eqnarray*}
  | Z^r_{ix} | + | Z^r_{1} - z^m_- | + | Z^r_{3} - z^m_+ | \lesssim \delta e^{-c(|x|+t)}
\end{eqnarray*}
and in $\Omega_i(i=1,3)$
\begin{eqnarray*}
  &&| Z^{c}_x | + | Z^{c} - z^m_\pm | \lesssim \delta e^{ -c(|x|+t)},\\[2mm]
  &&| Z^r_{ix} | + | Z^r_{1} - z^m_- | + | Z^r_{3} - z^m_+ | \lesssim \delta e^{ -c(|x|+t)}.
\end{eqnarray*}
$(iv)$ For the rarefaction wave $z^r_i\big(\frac{x}{t}\big)(i=1,3)$, it holds
\begin{eqnarray*}
  \lim\limits_{t\rightarrow \infty}\sup\limits_{x\in {\bf{R}}}|Z^r_i(x,t)-z^r_i(\tfrac{x}{t})|=0.
\end{eqnarray*}
\end{Lemma}

Setting $Z(x,t) := (V, U, \Theta)(x,t)$ is
 \begin{eqnarray}\label{2.20}
\left\{
\begin{array}{ll}
  V(x,t) = V^r_1(x,t) + V^{c}(x,t) + V^r_3(x,t) - v^m_- - v^m_+, \\[2mm]
  U(x,t) = U^r_1(x,t) + U^{c}(x,t) + U^r_3(x,t) - 2 u^m, \\[2mm]
  \Theta(x,t) = \Theta^r_1 (x,t) + \Theta^{c} (x,t) + \Theta^r_3 (x,t) - \theta^m_- - \theta^m_+.
\end{array}
\right.
\end{eqnarray}

Then our second main result is as follow:
\begin{Theorem}
For any given $z_-$, assume that $z_+ \in R_1CR_3(z_-)\subseteq \Omega_-$ with
$|\theta_+-\theta_-|\leq \delta_1$. There exist positive constants $\epsilon_2$
and $\delta_2(\leq \min\{\overline{\delta},\delta_1\})$, such that if $\delta<\delta_2$ and
the initial data satisfies
\begin{eqnarray}\label{2.21}
  \| (v_0(\cdot) - V(\cdot,0), u_0(\cdot) - U(\cdot,0), \theta_0(\cdot) - \Theta(\cdot,0)) \|_2 \leq \epsilon_2,
\end{eqnarray}
then the Cauchy problem (\ref{1.1}), (\ref{1.2}) admits a unique global
solution $(v, u, \theta)(x,t)$ satisfying
$$( v-V, u-U, \theta-\Theta )\in X([0, +\infty))$$
 and
\begin{eqnarray}\label{2.22}
\lim \limits_{t \rightarrow \infty} \sup \limits_{x \in {\mathbb{R}}}
\left\{
\begin{array}{ll}
  | ( v - v^r_1 - V^{c} - v^r_3 + v^m_- + v^m_+ )(x,t) |\\[2mm]
  | ( u - u^r_1 - U^{c} - u^r_3 + 2u^m )(x,t) |\\[2mm]
  | ( \theta - \theta^r_1 - \Theta^{c} - \theta^r_3 + \theta^m_- + \theta^m_+ )(x,t) |
\end{array}
\right\}
=0
\end{eqnarray}
where the solution space $X(0,t)$ is defined in (\ref{3.5}).
\end{Theorem}
\begin{Remark}
 if  $z^m_-=z^m_+=(v_m,u_m,\theta_m)$,
we can also get the stability  of the composition of two rarefaction waves.
\end{Remark}

\section{Stability Analysis}
\setcounter{equation}{0}
In this section, we will prove the asymptotic behavior of the solution for nonvisous system (\ref{1.1}), (\ref{1.3}).
If $z^m_\pm=z_\pm$, Theorem 2.2 will coincide with the result of Theorem 2.1, here we will show the stability of the composition wave only.

\subsection{Reformed the System}
 $Z(x,t):= (V, U, \Theta)(x,t)$ defined in (\ref{2.20})
satisfies
\begin{eqnarray}\label{3.1}
\left\{
\begin{array}{ll}
  V_t - U_x = 0, \\[2mm]
  U_t + P_x = -R_1, \\[2mm]
  \tfrac{R}{\gamma-1} \Theta_t + P U_x = \kappa \big( \frac{\Theta_x}{V} \big)_x - R_2,
\end{array}
\right.
\end{eqnarray}
where
\begin{eqnarray}\label{3.2}
\begin{aligned}
  P := & \frac{R \Theta}{V}, \quad P_i = \frac{R \Theta^r_i}{V^r_i} \ (i=1,3),\\[2mm]
  R_1 := & -(P - P_1 - P_3 - p^c)_x + U^{c}_t := R^1_1 + U^{c}_t, \\[2mm]
  R_2 := & \big\{ (p_m - P) U^{c}_x + (P_1 - P) U^r_{1x} + (P_3 - P) U^r_{3x} \big\} \\[2mm]
  & + \kappa \bigg\{ \Big( \frac{\Theta_x}{V} \Big)_x - \Big( \frac{\Theta^{c}_x}{V^{c}} \Big)_x \bigg\} := R^1_2 + R^2_2.
 \end{aligned}
\end{eqnarray}
Let  the perturbation function is
\begin{eqnarray*}
(\phi, \psi, \xi):=(v, u, \theta) - (V, U, \Theta),
\end{eqnarray*}
then the reformed equation is
 \begin{eqnarray}\label{3.3}
\left\{
\begin{array}{ll}
  \phi_t - \psi_x = 0, \\[2mm]
  \psi_t + \Big( \frac{R \xi}{v} - \frac{P \phi}{v} \Big)_x = R_1,\\[2mm]
  \tfrac{R}{\gamma-1} \xi_t + p \psi_x + (p-P) U_x = \kappa \Big( \frac{\xi_x}{v} - \frac{\Theta_x \phi}{v V} \Big)_x + R_2
\end{array}
\right.
\end{eqnarray}
with the initial data
\begin{eqnarray}\label{3.4}
  (\phi, \psi, \xi)(x, 0)& = & (\phi_0, \psi_0, \xi_0)(x)\\[2mm]
  & = & (v_0(x)-V(x,0), u_0(x)-U(x,0), \theta_0(x)-\Theta(x,0)) \,. \nonumber
\end{eqnarray}
And the solution space is defined as
\begin{eqnarray}\label{3.5}
\begin{aligned}
  X(0,t) := & \big\{ ( \phi,\psi,\xi) | (\phi,\psi,\xi) \in C([0,t),H^2({\mathbb{R}}) ),\\[2mm]
  & ( \phi,\psi)_x \in L^2( [0,t),H^1( {\mathbb{R}} ) ), \xi_x \in L^2 ( [0,t), H^2({\mathbb{R}}) ) \big\}\,.
\end{aligned}
\end{eqnarray}
The local existence is  known in \cite{Nishihara-Yang-Zhao_SIAMJMA_2004}
\begin{Proposition}(Local existence)
Under the assumptions stated in Theorem 2.1, the Cauchy problem (\ref{3.3}), (\ref{3.4})
admits a unique smooth solution $(\phi,\psi,\xi)(x,t) \in X(0,t_1)$ for some sufficient small
$t_1>0$, and $(\phi,\psi,\xi)(x,t) $ satisfies
  \begin{eqnarray}\label{3.6}
  \sup \limits_{0 \leq t \leq t_1} \| (\phi,\psi,\xi)(t) \|^2_2 \leq 2 \| (\phi_0,\psi_0,\xi_0) \|^2_2 \,.
\end{eqnarray}
\end{Proposition}
Suppose that  $ (\phi,\psi,\xi)(x,t) $ has been extended to the time $t > t_1$, we want to get the following a priori estimates to obtain a global solution.
\begin{Proposition}(A prior estimates)
Under the assumptions stated in Theorem 2.1, there exist positive constants
$\epsilon_2 \leq1$, $\delta_2 \leq \min \{ \delta_1, \overline{\delta}, 1 \}$
and C, such that  $(\phi,\psi,\xi) \in X([0,t])$ for some $t>0$ satisfying
\begin{eqnarray}\label{3.7}
\begin{aligned}
  &N(t) = \sup \limits_{0 \leq \tau \leq t} \| (\phi,\psi,\xi)(\tau) \|_2 \leq \epsilon_2 \, , \\[2mm]
  &\delta = |\theta_- - \theta_+ | < \delta_2 \, ,
\end{aligned}
\end{eqnarray}
it follows the estimate
\begin{eqnarray}\label{3.8}
\begin{aligned}
  &\sup \limits_{0 \leq \tau \leq t} \| (\phi,\psi,\xi)(\tau) \|^2_2 +
  \int^t_0 \big( \| (\phi_x,\psi_x)(\tau) \|^2_1 + \| \xi_x(\tau) \|^2_2 \big) d \tau \\[2mm]
  \lesssim & \| (\phi_0,\psi_0,\xi_0) \|^2_2 + \delta_2 \, .
\end{aligned}
\end{eqnarray}
\end{Proposition}

Once Proposition 3.2 is proved, we can extend the unique local solution $(\phi,\psi,\xi)(x,t) $
 obtained in Proposition 3.1 to $t=\infty$, moreover, estimate (\ref{3.8}) implies that
\begin{eqnarray}\label{3.9}
  \int^\infty_0 \big( \| (\phi_x,\psi_x,\xi_x)(t) \|^2 + \big| \frac{d}{dt} \| (\phi_x,\psi_x,\xi_x)(t) \|^2 \big| \big) d \tau
  \lesssim + \infty \, ,
\end{eqnarray}
which together with  Sobolev inequality easily leads to the asymptotic
behavior  (\ref{2.22}), this  concludes the proof of Theorem 2.1.
In the rest of this section, our main task is to show this a prior estimates.

\subsection{A Priori Estimates}
At first, we show the basic estimates.
\begin{Lemma}
Under the assumptions in proposition 3.2, we have
\begin{eqnarray}\label{3.10}
\begin{aligned}
  &\| (\phi,\psi,\xi)(t) \|^2 + \int^t_0 \int_{\mathbb{R}} \big( (|U^r_{1x}|+|U^r_{3x}|) (\phi^2+\xi^2) + \xi^2_x \big) dx d \tau \\[2mm]
  \lesssim & \| (\phi_0,\psi_0,\xi_0) \|^2 + \delta + \delta^{\frac{1}{2}} \int^t_0 \| (\phi_x,\psi_x)(\tau) \|^2 d\tau \\[2mm]
  & + \delta \int^t_0 \frac{1}{1+\tau} \int_{\mathbb{R}} (\phi^2+\xi^2) e^{\frac{-c x^2}{1+\tau}} dx d\tau.
\end{aligned}
\end{eqnarray}
\end{Lemma}
\begin{proof}
The first step, by using lemma 2.1, we need  get some estimates of $R_1$ and $R_2$. Since $R\Theta^{c}=p_m V^c$,
direct calculation yields that
\begin{eqnarray}\label{3.11}
  R^1_1 & = & R \Big( \frac{\Theta^r_1}{V^r_1} + \frac{\Theta^r_3}{V^r_3} + \frac{\Theta^{cd}}{V^{cd}} - \frac{\Theta}{V} \Big)_x \nonumber\\[2mm]
  & = & R \Theta^r_{1x} ( (V^r_1)^{-1} - V^{-1}) + R \Theta^r_{3x}( (V^r_3)^{-1}-V^{-1} ) \\[2mm]
  && + R \Theta^{cd}_x ( (V^{cd})^{-1} - V^{-1} ) + R V^r_{1x} \Big( \frac{\Theta}{V^2} - \frac{\Theta^r_1}{V^{r2}_1} \Big) \nonumber \\[2mm]
  && + R V^r_{3x} \Big( \frac{\Theta}{V^2} - \frac{\Theta^r_3}{V^{r2}_3} \Big) + R V^{cd}_{x} \Big( \frac{\Theta}{V^2} - \frac{\Theta^{c}}{V^{c2}} \Big). \nonumber
\end{eqnarray}
It is easy to compute
\begin{eqnarray*}
  | \Theta^r_{1x} ((V^r_1)^{-1} - V^{-1}) |
  \lesssim | \Theta^r_{1x} | ( |V^r_{3} - v^m_+| + |V^{c}-v^m_-| )
  \lesssim \delta e^{-c(|x|+t)}
\end{eqnarray*}
and we can treat the other terms on the righthand side of (\ref{3.11}) in the same way to
obtain
\begin{eqnarray}\label{3.12}
  |R^1_1| \lesssim \delta e^{-c(|x|+t)}
\end{eqnarray}
and
\begin{eqnarray}\label{3.13}
  |R_1| \lesssim |R^1_1| + |U^{c}_t| \lesssim \delta e^{-c(|x|+t)} + \tfrac{\delta}{(1+t)^\frac{3}{2}} e^{\frac{-cx^2}{1+t}}.
\end{eqnarray}
Similarly we have
\begin{eqnarray*}
  |R^1_2|  \lesssim \delta e^{-c(|x|+t)}.
\end{eqnarray*}
Since
\begin{eqnarray*}
  R^2_2 = \kappa \Big( \frac{\Theta^r_{1x}}{V} + \frac{\Theta^r_{3x}}{V} \Big)_x
  + \kappa \Big( \frac{\Theta^{cd}_x}{V} - \frac{\Theta^{cd}_x}{V^{cd}} \Big)_x := R^2_{21} + R^2_{22}
\end{eqnarray*}
and
\begin{eqnarray}\label{3.14}
  R^2_{21} & \lesssim & \Big( |\Theta^r_{1xx}| + |\Theta^r_{3xx}| + |\Theta^r_{1x}| |V^r_{1x}| + |\Theta^r_{3x}| |V^r_{3x}| \Big)\\[2mm]
  && + |\Theta^r_{1x}| (|V^r_{3x}| + |V^{cd}_x|) + |\Theta^r_{3x}| (|V^r_{1x}| + |V^{cd}_x|), \nonumber
\end{eqnarray}
it follows from (\ref{2.6}) and Lemma 2.1 that
\begin{eqnarray}\label{3.15}
\begin{aligned}
  |R^2_{21}| \lesssim & \delta^\frac{1}{8} (1+t)^{-\frac{7}{8}},\\[2mm]
  |R^2_{22}| \lesssim & \Big( |\Theta^{c}_{xx}| + |\Theta^{c}_{x}| |V^{c}_x| \Big) \Big( |V^r_3-v^m_3| + |V^r_1-v^m_-| \Big) \\[2mm]
  & + |\Theta^{c}_{x}| (|V^r_{1x}| + |V^r_{3x}|)
  \lesssim \delta e^{-c(|x|+t)}.
\end{aligned}
\end{eqnarray}
Thus we get
\begin{eqnarray}\label{3.16}
  |R_2| \lesssim \delta^\frac{1}{8} (1+t)^{-\frac{7}{8}}\, .
\end{eqnarray}
And by the direct calculation, we can also obtain
\begin{eqnarray}\label{3.17}
\begin{aligned}
  & |(R_{1x}, R_{1xx})(x,t)| \lesssim \delta e^{-c(|x|+t)} + \frac{\delta}{(1+t)^\frac{3}{2}} e^{\frac{-cx^2}{1+t}}, \\[2mm]
  &|(R_{2x}, R_{2xx})(x,t)| \lesssim \delta^\frac{1}{8} (1+t)^{-\frac{7}{8}}.
\end{aligned}
\end{eqnarray}
Then multiplying $(\ref{3.3})_1$ by $-R\Theta \big(\frac{1}{v}-\frac{1}{V}\big)$,
 $(\ref{3.3})_2$ by $\psi$ and $(\ref{3.3})_3$ by $\frac{\xi}{\theta}$, respectively,
and adding the results together, we get
\begin{eqnarray}\label{3.18}
\begin{aligned}
  &\bigg\{ R \Theta \Phi \Big( \frac{v}{V} \Big) + \tfrac{1}{2} \psi^2 + \tfrac{R\Theta}{\gamma-1} \Phi \Big( \frac{\theta}{\Theta} \Big) \bigg\}_t + \frac{\kappa}{v\theta} \xi^2_x + H_{1x} + Q_1 + Q_2 \\[2mm]
  = & R_1 \psi + R_2 \frac{\xi}{\theta},
  \end{aligned}
\end{eqnarray}
where
\begin{eqnarray}\label{3.19}
\begin{aligned}
  H_1& := (p-P) \psi - \frac{\kappa\xi}{\theta} \Big( \frac{\xi_x}{v} - \frac{\Theta_x\phi}{vV} \Big), \\[2mm]
  Q_1& := - R \Theta_t \Phi \Big( \frac{v}{V} \Big) + \frac{P U_x}{v V} \phi^2 + \frac{R \Theta_t}{\gamma-1} \Phi \Big( \frac{\Theta}{\theta} \Big) + \frac{\xi}{\theta} (p-P) U_x,\\[2mm]
  Q_2& := - \kappa \frac{\theta_x}{\theta^2 v} \xi \xi_x - \kappa \frac{\xi_x \phi}{\theta v V} \Theta_x + \kappa \frac{\theta_x \xi \phi}{\theta^2 v V} \Theta_x,
 \end{aligned}
\end{eqnarray}
here
$$\Phi(s)=s-1-\ln s.$$
In the process of the calculation, we have used the following equalities
\begin{eqnarray}\label{3.20}
\begin{aligned}
  - R \Theta \Big( \frac{1}{v} - \frac{1}{V} \Big) \phi_t
  & = \bigg\{ R \Theta \Phi \Big( \frac{v}{V} \Big) \bigg\}_t + \frac{P U_x}{v V} \phi^2 - R \Theta_t \Phi \Big( \frac{v}{V} \Big), \\[2mm]
  \frac{R \xi_t}{\gamma-1} \frac{\xi}{\theta}
  & = \bigg\{ \frac{R \Theta}{\gamma-1} \Phi \Big( \frac{\theta}{\Theta} \Big) \bigg\}_t + \frac{R \Theta_t}{\gamma-1} \Phi \Big( \frac{\Theta}{\theta} \Big)
\end{aligned}
\end{eqnarray}
and the equations in (\ref{3.3}).

Since the composite waves affected each other, we will try to divide it the viscous contact wave $Z^c$ and
 viscous rarefaction wave $Z^r_i(i=1,3)$,
\begin{eqnarray}\label{3.21}
\begin{aligned}
  - R \Theta_t = & (\gamma-1) (P^r_1 U^r_{1x} + P^r_3 U^r_{3x} ) + (\gamma-1)\bigg( p^m U^{c}_{x} - \kappa \Big( \frac{\Theta^{c}_{x}}{V^c} \Big)_x \bigg) \\[2mm]
  = & (\gamma-1) P ( U^r_{1x} + U^r_{3x} ) + (\gamma-1) (P^r_1-P) U^r_{1x} \\[2mm]
  & + (\gamma-1) (P^r_3-P) U^r_{3x} + (\gamma-1) \bigg( p^m U^{c}_{x} - \kappa \Big( \frac{\Theta^{c}_{x}}{V^c} \Big)_x \bigg),
\end{aligned}
\end{eqnarray}
thus

\begin{eqnarray}\label{3.22}
\begin{aligned}
  Q_1 = & - R \Theta_t \Phi \Big( \frac{v}{V} \Big) + \frac{R \Theta_t}{\gamma-1} \Phi \Big( \frac{\Theta}{\theta} \Big) + \frac{ P (U_{1x} + U_{3x} + U^c_{x}) }{v V} \phi^2 \\[2mm]
  & + \frac{\xi}{\theta} (p-P) ( U_{1x}+U_{3x}+U^c_{x} ) \\[2mm]
  = &: ( |U^r_{1x}| + |U^r_{3x}| ) Q_{11} + Q_{12},
\end{aligned}
\end{eqnarray}
where
\begin{eqnarray}\label{3.23}
\begin{aligned}
  Q_{11} := & (\gamma-1) P \Phi \Big( \frac{v}{V} \Big) - P \Phi \Big( \frac{\Theta}{\theta} \Big) + \frac{P}{v V} \phi^2 + \frac{\xi}{\theta} (p-P) \\[2mm]
  = & P \bigg( \frac{\theta V}{\Theta v} - 1 + \gamma \Big( \frac{\theta}{\Theta} \Big) -
  \Big( \log \frac{\theta}{\Theta} + (\gamma-1) \log \frac{v}{V} \Big) \bigg) \\[2mm]
  = & P \bigg( \Phi \Big( \frac{\theta V}{\Theta v} \Big) + \gamma \Phi \Big( \frac{v}{V} \Big) \bigg) \geq C ( \phi^2 + \xi^2 )
\end{aligned}
\end{eqnarray}
and
\begin{eqnarray}\label{3.24}
\begin{aligned}
  Q_{12} := & U^{c}_{x} \bigg( \frac{P \phi^2}{v V} + (\gamma-1) p_m \Phi \Big( \frac{v}{V} \Big) + p_m \Phi \Big( \frac{\Theta}{\theta} \Big) + \frac{\xi}{\theta} (p-P) \bigg) \\[2mm]
  & + (P^r_1-P) U^r_{1x} \Big[ (\gamma-1) \Phi \Big( \frac{ v}{V} \Big) - \Phi \Big( \frac{v}{V} \Big) \Big] \\[2mm]
  & - \kappa \Big( \frac{\Theta^{c}_{x}}{V^c} \Big)_x \Big[ (\gamma-1) \Phi \Big( \frac{v}{V} \Big) - \Phi \Big( \frac{v}{V} \Big) \Big ] \\[2mm]
  & + (P^r_3-P) U^r_{3x} [ (\gamma-1) \Phi \Big( \frac{v}{V} \Big) - \Phi \Big( \frac{v}{V} \Big) \Big] \\[2mm]
  \lesssim & | (U^{c}_{x}, \Theta^{c}_{xx}, \Theta^{c}_{x} V^{c}_{x}) | (\phi^2+\xi^2) + \delta (U^r_{1x} + U^r_{3x}) (\phi^2 + \xi^2) \, .
\end{aligned}
\end{eqnarray}
Meanwhile,
\begin{eqnarray}\label{3.25}
  |Q_2| \lesssim \big( N(t) + \delta + \tfrac{1}{8} \big) \xi^2_x + \Theta^2_{x} \big( \phi^2 + \xi^2 \big),
\end{eqnarray}
where
\begin{eqnarray}\label{3.26}
\begin{aligned}
\Theta^2_{x}
\lesssim & ({\Theta^r_{1x}}^2+{\Theta^r_{3x}}^2)+({\Theta^c_x}^2)\\[2mm]
\lesssim & {\Theta^r_{1x}}^2+{\Theta^r_{3x}}^2+
\frac{\delta}{1+t}e^{\frac{-cx^2}{1+t}} \, .
\end{aligned}
\end{eqnarray}
Therefore, after integrating (\ref{3.19}) on $[0, t] \times \mathbb{R}$ and using the above estimates, we obtain
\begin{eqnarray}\label{3.27}
\begin{aligned}
  & \| (\phi,\psi,\xi)(t) \|^2 + \int^t_0 \int_{\mathbb{R}} \big( ( |U^r_{1x}| + |U^r_{3x}| ) (\phi^2+\xi^2) + \xi^2_x \big) dx d\tau \\[2mm]
  \lesssim & \| (\phi_0,\psi_0,\xi_0) \|^2 + \int^t_0 \frac{\delta}{1+\tau} \int_{\mathbb{R}} (\phi^2+\xi^2) e^{ \frac{-cx^2}{1+\tau} } dx d\tau \\[2mm]
  & + \int^t_0 \int_{\mathbb{R}} ( {\Theta^r_{1x}}^2 + {\Theta^r_{3x}}^2 ) (\phi^2+\xi^2) dx d\tau
  + \int^t_0 \int_{\mathbb{R}} ( |\psi||R_1| +  |\xi||R_2| ) dx d\tau.
\end{aligned}
\end{eqnarray}
Noticing that $\| (\Theta^r_{1x}, \Theta^r_{3x} ) \|_{\infty} \lesssim \delta^\frac{1}{8} (1+t)^{-\frac{7}{8}}$, one can easily get
\begin{eqnarray}\label{3.28}
  && \int^t_0 \int_{\mathbb{R}} ( {\Theta^r_{1x}}^2 + {\Theta^r_{3x}}^2 ) ( \phi^2+\xi^2 ) dx d\tau \nonumber\\[2mm]
  & \lesssim & \delta^\frac{1}{4} \int^t_0 \| (\phi,\xi) \|^2_{\infty} (1+\tau)^{-\frac{7}{4}} d\tau \\[2mm]
  & \lesssim & \delta^\frac{1}{4} \int^t_0 \| (\phi,\xi) \| \| (\phi_x,\xi_x) \| (1+\tau)^{-\frac{7}{4}} d\tau \nonumber \\[2mm]
  & \lesssim & \delta^\frac{1}{4} \int^t_0 \| (\phi_x,\xi_x) \|^2 d\tau + \delta^\frac{1}{4} \int^t_0 \| (\phi,\xi) \|^2 (1+\tau)^{-\frac{7}{2}} d\tau \nonumber \\[2mm]
  & \lesssim & \delta^\frac{1}{4} \int^t_0 \| (\phi_x,\xi_x) \|^2 d\tau + \delta^\frac{1}{4}. \nonumber
\end{eqnarray}
For the last term in (\ref{3.27}), we have the following estimate
\begin{eqnarray}\label{3.29}
  && \int^t_0 \int_{\mathbb{R}} ( |\psi||R_1|+  |\xi||R_2| ) dx d\tau \nonumber \\[2mm]
  & \lesssim & \delta \int^t_0 \| \psi\ |_{\infty} ( \int_{\mathbb{R}} e^{-c(|x|+\tau)} dx + \int_{\mathbb{R}} \frac{\delta}{(1+\tau)^\frac{3}{2}} e^{\frac{-cx^2}{1+\tau}} dx ) d\tau  \nonumber \\[2mm]
  && + \delta^\frac{1}{8} \int^t_0 \| \xi \|_{\infty} (1+\tau)^{-\frac{7}{8}} d\tau \\[2mm]
  & \lesssim  & \delta \int^t_0 \| \psi \|^\frac{1}{2} \| \psi_x \|^\frac{1}{2} (1+\tau)^{-1} d\tau
  + \delta^\frac{1}{8}\int^t_0 \| \xi \|^\frac{1}{2} \| \xi_x \|^\frac{1}{2} (1+\tau)^{-\frac{7}{8}} d\tau \nonumber \\[2mm]
  & \lesssim & \delta \int^t_0 \| (\psi_x,\xi_x) \|^2 d\tau + \delta^\frac{1}{8} \int^t_0 \| (\psi,\xi) \|^\frac{2}{3} (1+\tau)^{-\frac{7}{6}} d\tau \nonumber\\[2mm]
  & \lesssim &  \delta \int^t_0 \| (\psi_x,\xi_x) \|^2 d\tau + \delta. \nonumber
\end{eqnarray}
Inserting (\ref{3.28}) and (\ref{3.29}) into (\ref{3.27}), we can get (\ref{3.10}) and this  completes the proof of Lemma 3.1.
\end{proof}
Then we claim that there exists some positive constant $C$ such that
\begin{Lemma}(\cite{Huang-Li-Matsumura-ARMA-2010})
Under the assumptions in proposition 3.2, we have
\begin{eqnarray}\label{3.30}
\begin{aligned}
  & \int^t_0 \frac{1}{1+\tau} \int_{\mathbb{R}} ( \phi^2 + \psi^2 + \xi^2 ) e^{-\frac{cx^2}{1+\tau}} dx d\tau \\[2mm]
  \lesssim & 1 + \int^t_0 \| (\phi_x,\psi_x,\xi_x) \|^2 d\tau + \int^t_0 \int_{\mathbb{R}} ( |U^r_{1x}| + |U^r_{3x}| ) (\phi^2+\xi^2) dx d\tau.
\end{aligned}
\end{eqnarray}
\end{Lemma}
For $\alpha>0$, we define the following heat kernel $\omega(x,t)$ which will play an essential role in the later estimates. We define
\begin{eqnarray*}
  \omega(x,t) = (1+t)^{-\frac{1}{2}} e^{-\frac{\alpha x^2}{1+t}} \, , \quad
  g(x,t) = \int_{-\infty}^x \omega(y,t) dy \, .
\end{eqnarray*}
It is easy to check that
\begin{eqnarray*}
  4 \alpha g_t = g_{xx} \, , \quad
  \| g(\cdot,t) \|_{L^\infty} = \sqrt{\pi} \alpha^{-\frac{1}{2}}\, .
\end{eqnarray*}
The proof (\ref{3.30}) is divided into two parts:
\begin{eqnarray}\label{3.31}
\begin{aligned}
  & \int^t_0 \int_{\mathbb{R}} \omega^2 ( (R\xi-P\phi)^2 + \psi^2 ) dx d\tau \\[2mm]
  \lesssim & 1 + \int^t_0 \| ( \phi_x,\psi_x,\xi_x ) (\tau) \|^2 d\tau + \delta \int^t_0 \int_{\mathbb{R}} \omega^2 (\phi^2+\psi^2+\xi^2) dx d\tau \\[2mm]
  & + \delta \int^t_0 \int_{\mathbb{R}} ( |U^r_{1x}| + |U^r_{3x}| ) ( \phi^2+\psi^2+\xi^2 ) dx d\tau
\end{aligned}
\end{eqnarray}
and for any $\eta>0$,
\begin{eqnarray}\label{3.32}
\begin{aligned}
  & \int^t_0 \int_{\mathbb{R}} \omega^2 ( R\xi+(\gamma-1)P\phi )^2 dx d\tau \\[2mm]
  \lesssim & 1 + \int^t_0 \| (\phi_x,\psi_x,\xi_x) (\tau) \|^2 d\tau + (\delta+\eta) \int^t_0 \int_{\mathbb{R}} \omega^2 (\phi^2+\psi^2+\xi^2) dx d\tau, \\[2mm]
  & + \delta \int^t_0 \int_{\mathbb{R}} ( |U^r_{1x}| + |U^r_{3x}| ) (\phi^2+\psi^2+\xi^2) dx d\tau  \, .
\end{aligned}
\end{eqnarray}
Here we used the same method as in \cite{Huang-Li-Matsumura-ARMA-2010}, the only difference lied in the
loss of diffusion term of fluid velocity, but that's not the main term. We omit the detail for simplify.

We now turn to obtain the higher order estimates, further calculation yields the following Lemma

\begin{Lemma}
Under the assumptions in proposition 3.2, we have
\begin{eqnarray}\label{3.33}
\begin{aligned}
&\|(\phi_x,\psi_x,\xi_x)(t)\|^2+
\int^t_0\|\xi_{xx}(\tau)\|^2d\tau \\[2mm]
\lesssim &\|(\phi_{0x},\psi_{0x},\xi_{0x})\|^2+\delta+(\delta+N(t)) \int^t_0\|(\phi_x,\psi_x)(\tau)\|^2_1d\tau \, .
\end{aligned}
\end{eqnarray}
\end{Lemma}
\begin{proof}
Multiplying $(\ref{3.3})_{1x}$ by $\frac{P}{v}\phi_x$, $(\ref{3.3})_{2x}$ by $\psi_x$ and $(\ref{3.3})_{3x}$ by $\frac{\xi_x}{\theta}$, respectively, and adding the resulting equations together, then we have
\begin{eqnarray}\label{3.34}
\begin{aligned}
  & \bigg\{ \frac{P}{2v} \phi_x^2 + \frac{\psi^2_x}{2} + \frac{R \xi^2_x}{2(\gamma-1)\theta} \bigg\}_t
  + \frac{\kappa}{v\theta} \xi^2_{xx} + H_{2x} + J_2 \\[2mm]
  = & R_{1x} \psi_{x} + R_{2x} \frac{\xi_{x}}{\theta} \, ,
 \end{aligned}
\end{eqnarray}
where
\begin{eqnarray}\label{3.35}
\begin{aligned}
  H_2 = & (p-P)_x \psi_x + \frac{\xi_x}{\theta} \bigg( (p-P) U_x + \Big( \frac{\kappa\xi_x}{v} - \frac{\kappa \Theta_x \phi}{vV} \Big)_x \bigg) \, , \\[2mm]
  J_2 = & \Big( \frac{\xi_x}{\theta} \Big)_x \bigg( \Big( \frac{\kappa \xi_x}{v} - \frac{\kappa \Theta_x \phi}{v V} \Big)_x - (p-P) U_x \bigg) - \frac{\kappa}{v \theta} \xi^2_{xx} -  \Big( \frac{P}{2v} \Big)_t \phi^2_x \\[2mm]
  & - \frac{R}{2 (\gamma-1)} \Big( \frac{1}{\theta} \Big)_t \xi^2_x - \Big( \frac{R}{v} \Big)_x \xi \psi_{xx} + \Big( \frac{P}{v} \Big)_x \phi \psi_{xx} + p_x \psi_x \tfrac{\xi_x}{\theta} \\[2mm]
  = & O(1) \big( N(t) + \delta + \eta \big) | (\phi_x,\xi_x,\psi_{xx},\xi_{xx}) |^2 + | (V_x,U_x,\Theta_x,\Theta_{xx}) |^2 (\phi^2+\xi^2)\, .
\end{aligned}
\end{eqnarray}
After integrating (\ref{3.34}) on $[0, t] \times \mathbb{R}$, we get
\begin{eqnarray}\label{3.36}
\begin{aligned}
  & \| (\phi_x,\psi_x,\xi_x)(t) \|^2 + \int^t_0 \| \xi_{xx}(\tau) \|^2 d\tau  \\[2mm]
  \lesssim & \| (\phi_{0x}, \psi_{0x}, \xi_{0x}) \|^2
  + ( \delta + N(t) + \eta ) \int^t_0 \| (\phi_x, \psi_{x}, \xi_x)(\tau) \|^2_1 dx d\tau  \\[2mm]
  & + \int^t_0 \int_{\mathbb{R}} ( |\Theta_{xx}| + |\Theta_x| )^2 ( \phi^2 + \xi^2 ) dx d\tau
  + \int^t_0 \int_{\mathbb{R}} ( | R_{1x} \psi_{x} | + | R_{2x} \xi_{x} |) dx d\tau\, ,
\end{aligned}
\end{eqnarray}
where $\eta>0$ is a constant suitably small, the last two terms on the right hand side of (\ref{3.36}) can be treated similarly as  (\ref{3.28}) and (\ref{3.29}), respectively. Then, with the help of the results of Lemma 3.1 and Lemma 3.2 we can deduce the estimate (\ref{3.33}), and this  completes the proof of Lemma 3.3.
\end{proof}

\begin{Lemma}
Under the assumption in proposition 3.2, we have
\begin{eqnarray}\label{3.37}
\begin{aligned}
  & \| (\phi_{xx}, \psi_{xx}, \xi_{xx})(t) \|^2 +
  \int^t_0 \| \xi_{xxx}(\tau) \|^2 d\tau \\[2mm]
  \lesssim & \| (\phi_{0xx}, \psi_{0xx}, \xi_{0xx}) \|^2 + \delta + ( \delta + N(t) ) \int^t_0 \| (\phi_x,\psi_x)(\tau) \|^2_1 d\tau\, .
\end{aligned}
\end{eqnarray}
\end{Lemma}

\begin{proof}

Multiplying $(\ref{3.3})_{1xx}$ by $\frac{P}{v}\phi_{xx}$,
 $(\ref{3.3})_{2xx}$ by $\psi_{xx}$ and $(\ref{3.3})_{3xx}$ by $\frac{\xi_{xx}}{\theta}$, respectively,
and adding the results together, it is easily to get
 \begin{eqnarray}\label{3.38}
 \begin{aligned}
  & \bigg\{ \frac{P}{2v} \phi_{xx}^2 + \frac{\psi^2_{xx}}{2} + \frac{R \xi^2_{xx}}{2(\gamma-1) \theta}
  \bigg\}_t + \frac{\kappa}{v\theta} \xi^2_{xxx} + H_{3x} + J_3 \\[2mm]
  = & R_{1xx} \psi_{xx} + R_{1xx} \frac{\xi_{xx}}{v}\, ,
\end{aligned}
\end{eqnarray}
where
\begin{eqnarray}\label{3.39}
\begin{aligned}
  H_3 = & (p-P)_{xx} \psi_{xx} + \frac{\xi_{xx}}{\theta} \bigg( (p-P)U_x - \Big( \frac{\kappa\xi_x}{v} - \frac{\kappa \Theta_x \phi}{vV} \Big)_x \bigg)_{x}, \\[2mm]
  J_3 = & \Big( \frac{\xi_{xx}}{\theta} \Big)_x
  \bigg( \Big( \frac{\kappa \xi_x}{v} - \frac{\kappa \Theta_x \phi}{vV} \Big)_x - (p-P) U_x \bigg)_x
  - \frac{\kappa}{v \theta} \xi^2_{xxx} \\[2mm]
  & - \Big( \frac{P}{2v} \Big)_t \phi_{xx}^2 - \Big( \frac{R}{2 (\gamma-1) \theta} \Big)_t \xi^2_{xx} + (2 p_x \psi_{xx} + p_{xx} \psi_x ) \frac{\xi_{xx}}{\theta} \\[2mm]
  & + 3 \bigg\{ \Big( \frac{R}{v} \Big)_x \xi_x \bigg\}_x \psi_{xx} - 3 \bigg\{ \Big( \frac{P}{v} \Big)_x \phi_x \bigg\}_x \psi_{xx} + 2 J^1_3 + J^2_3\, .
\end{aligned}
\end{eqnarray}
Here $J^1_3$, $J^2_3$ are the following equalities
\begin{eqnarray*}
\begin{aligned}
  J^1_3 :&= \psi_{xxx} \bigg( \Big( \frac{P}{v} \Big)_x \phi_x - \Big( \frac{R}{v} \Big)_x \xi_x \bigg) \, , \\[2mm]
  J^2_3 :&= \psi_{xxx} \bigg( \Big( \frac{P}{v} \Big)_{xx} \phi - \Big( \frac{R}{v} \Big)_{xx} \xi \bigg) \, .
\end{aligned}
\end{eqnarray*}
Meanwhile, we can get
\begin{eqnarray}\label{3.40}
\begin{aligned}
  J^1_3 = & \bigg\{ \psi_{xx} \bigg( \Big( \frac{P}{v} \Big)_x \phi_x - \Big( \frac{R}{v} \Big)_x \xi_x \bigg) \bigg\}_x
  - \psi_{xx} \bigg( \Big( \frac{P}{v} \Big)_x \phi_x - \Big( \frac{R}{v} \Big)_x \xi_x \bigg)_x \\[2mm]
  = & \bigg\{ \psi_{xx} \bigg( \Big( \frac{P}{v} \Big)_x \phi_x - \Big( \frac{R}{v} \Big)_x \xi_x \bigg) \bigg\}_x
  + O(1) ( N(t) + \delta ) | ( \phi_x, \xi_x, \phi_{xx}, \psi_{xx}, \xi_{xx} ) |^2
\end{aligned}
\end{eqnarray}
and
\begin{eqnarray}\label{3.41}
\begin{aligned}
  & \psi_{xxx} \Big( \frac{P}{v} \Big)_{xx} \phi \\[2mm]
  = & \psi_{xxx} \phi \bigg( \frac{P_{xx}}{v} + 2 P_{x} \Big( \frac{1}{v} \Big)_x + P \Big( \frac{2v^2_x}{v^3} - \frac{V_{xx}}{v^2} \Big) \bigg) - \frac{P\phi}{v^2} \phi_{xx} \psi_{xxx} \\[2mm]
  = & - \frac{P\phi}{v^2} \phi_{xx} \phi_{txx} + \bigg\{ \psi_{xx} \phi \bigg( \frac{P_{xx}}{v} + 2 P_{x} \Big( \frac{1}{v} \Big)_x + P \Big( \frac{2v^2_x}{v^3} - \frac{V_{xx}}{v^2} \Big) \bigg) \bigg\}_x \\[2mm]
  & - \psi_{xx} \bigg\{ \phi \bigg( \frac{P_{xx}}{v} + 2 P_{x} \Big( \frac{1}{v} \Big)_x + P \Big( \frac{2v^2_x}{v^3} - \frac{V_{xx}}{v^2} \Big) \bigg) \bigg\}_x \\[2mm]
  = & - \bigg\{ \frac{P\phi}{v^2} \frac{\phi^2_{xx}}{2} \bigg\}_t + \bigg\{ \psi_{xx} \phi \bigg( \frac{P_{xx}}{v} + 2 P_{x} \Big( \frac{1}{v} \Big)_x + P \Big( \frac{2v^2_x}{v^3} - \frac{V_{xx}}{v^2} \Big) \bigg) \bigg\}_x \\[2mm]
  & + O(1) ( N(t) + \delta ) | ( \phi_x, \xi_x, \phi_{xx}, \psi_{xx}, \xi_{xx} ) |^2 + | ( \Theta_x, \Theta_{xx} ) |^2 | ( \phi, \xi ) |^2 \, .
\end{aligned}
\end{eqnarray}
Similar to the estimate of $\psi_{xxx} \Big( \frac{P}{v} \Big)_{xx} \phi$, we have
\begin{eqnarray*}
\begin{aligned}
  \psi_{xxx} \Big( \frac{R}{v} \Big)_{xx} \xi
  = & \bigg\{ \frac{R \xi}{v^2} \frac{\phi^2_{xx}}{2} \bigg\}_t
  + \bigg\{ R \psi_{xx} \xi \Big( \frac{V_{xx}}{v^2} - \frac{2v^2_x}{v^3} \Big) \bigg\}_x + | ( \Theta_x, \Theta_{xx} ) |^2 | ( \phi, \xi ) |^2 \\[2mm]
  & + O(1) ( N(t) + \delta ) | ( \phi_x, \xi_x, \phi_{xx}, \psi_{xx}, \xi_{xx} ) |^2 \, .
\end{aligned}
\end{eqnarray*}
Therefore, it holds
\begin{eqnarray}\label{3.42}
\begin{aligned}
  J_3 = & \Big\{ \frac{R \xi}{v^2} \frac{\phi^2_{xx}}{2} - \frac{P\phi}{v^2} \frac{\phi^2_{xx}}{2} \Big\}_t
  + \bigg\{ \psi_{xx} \phi \bigg( \frac{P_{xx}}{v} + 2 P_{x} \Big( \frac{1}{v} \Big)_x + P \Big( \frac{2v^2_x}{v^3} - \frac{V_{xx}}{v^2} \Big) \bigg) \bigg\}_x \\[2mm]
  & + \bigg\{ \psi_{xx} R \xi \Big( \frac{V_{xx}}{v^2} - \frac{2v^2_x}{v^3} \Big) \bigg\}_x + \bigg\{ \psi_{xx} \bigg( \Big( \frac{P}{v} \Big)_x \phi_x - \Big( \frac{R}{v} \Big)_x \xi_x \bigg) \bigg\}_x \\[2mm]
  & + O(1) ( N(t) + \delta ) | ( \phi_x, \xi_x, \phi_{xx}, \psi_{xx}, \xi_{xx} ) |^2 + | ( \Theta_x, \Theta_{xx} ) |^2 | (\phi, \xi) |^2 \, .
\end{aligned}
\end{eqnarray}
After integrating (\ref{3.38}) on $[0, t] \times \mathbb{R}$, we get
\begin{eqnarray}\label{3.43}
\begin{aligned}
  & \| ( \phi_{xx}, \psi_{xx}, \xi_{xx} )(t) \|^2 + \int^t_0 \| \xi_{xxx}(\tau) \|^2 d\tau \\[2mm]
  \lesssim & \| ( \phi_{0}, \psi_{0}, \xi_{0} ) \|^2_2 + ( \delta + N(t) ) \int^t_0 \| ( \phi_x, \psi_x, \xi_x )(\tau) \|^2_1 d\tau \\[2mm]
  & + \int^t_0 \int_{\mathbb{R}} ( | \Theta_{xx} | + | \Theta_x | )^2 | ( \phi, \xi ) |^2 dx d\tau \\[2mm]
  & + \int^t_0 \int_{\mathbb{R}} | R_{1xx} \psi_{xx} + R_{2xx} \xi_{xx} | dx d\tau \, .
\end{aligned}
\end{eqnarray}
For the estimate of the last term in (\ref{3.43}), we have
\begin{eqnarray}\label{3.44}
\begin{aligned}
  & \int^t_0 \int_{\mathbb{R}} | R_{1xx} \psi_{xx} | dx d\tau \\[2mm]
  \lesssim & \delta \int^t_0 \big\| \psi_{xx} (\tau) \big\| \bigg\{ \Big( \int_{\mathbb{R}} e^{-2c|x|} e^{-2c\tau} dx \Big)^\frac{1}{2} + \Big( \frac{1}{(1+\tau)^3} e^{ \frac{-2cx^2}{1+\tau} } dx \Big) ^\frac{1}{2} \bigg\} d\tau \\[2mm]
  \lesssim & \delta \int^t_0 (1+t)^{-\frac{5}{4}} \| \psi_{xx} (\tau) \| d\tau
  \lesssim \delta \int^t_0 \| \psi_{xx}(\tau) \|^2 d\tau + \delta
\end{aligned}
\end{eqnarray}
and
\begin{eqnarray}\label{3.45}
  &&\int^t_0\int_{\bf R}|R_{2xx}\xi_{xx}|dxd\tau\nonumber\\[2mm]
  & \lesssim  & \delta^\frac{1}{8}\int^t_0\|\xi_{xx}\|_{\infty}(1+\tau)^{-\frac{7}{8}}d\tau \\[2mm]
  & \lesssim  & \delta^\frac{1}{8}\int^t_0\|\xi_{xx}\|^\frac{1}{2}\|\xi_{xxx}\|^\frac{1}{2}(1+\tau)^{-\frac{7}{8}}d\tau \nonumber\\[2mm]
  & \lesssim  & \delta^\frac{1}{8}\int^t_0 \|\xi_{xx}\|\|\xi_{xxx}\|d\tau + \delta^\frac{1}{8}\int^t_0(1+\tau)^{-\frac{7}{3}}d\tau
  \nonumber\\[2mm]
  &\lesssim & \delta^{\frac{1}{8}} \int^t_0 \|\xi_{xx}\|^2_1d\tau + \delta^{\frac{1}{8}} \,,\nonumber
\end{eqnarray}
then using the results of Lemma 3.1-Lemma 3.3, we can get (\ref{3.37}) and this  completes the proof of Lemma 3.4.
\end{proof}

Combining the results of Lemma 3.1-Lemma 3.4, we kown that
\begin{eqnarray}\label{3.46}
\begin{aligned}
  & \| ( \phi, \psi, \xi )(t) \|^2_2 + \int^t_0 \| \xi_x(\tau) \|^2_2 d\tau \\[2mm]
  \lesssim & \| ( \phi_0, \psi_0, \xi_0) \|^2_2 + \delta + ( \delta + N(t) + \eta ) \int^t_0 \| ( \phi_x, \psi_x ) (\tau) \|^2_1 d\tau \, .
\end{aligned}
\end{eqnarray}
Based on these analysis, we now deal with the last term on  the right hand in (\ref{3.46}).
\begin{Lemma}
Under the assumption in proposition 3.2, we have
\begin{eqnarray}\label{3.47}
  \int^t_0 \| ( \phi_x, \psi_x )(\tau) \|^2_1 d\tau
  \lesssim \| ( \phi_0, \psi_0, \xi_0 ) \|^2_2 + \delta \, .
\end{eqnarray}
\end{Lemma}
\begin{proof}
 Multiplying $(\ref{3.3})_2$ by $-\frac{P}{2}\phi_x$, and $(\ref{3.3})_3$ by $\psi_x$, respectively, and adding all the resultant equations, we have
\begin{eqnarray}\label{3.48}
\begin{aligned}
  & \Big\{ \tfrac{R}{\gamma-1} \xi \psi_x - \frac{P}{2} \phi_x \psi \Big\}_t + \Big\{ \frac{P}{2} \phi_t \psi - \tfrac{R}{\gamma-1} \xi \psi_t \Big\}_x + \frac{P^2}{2v} \phi^2_x + \frac{P}{2} \psi^2_x \\[2mm]
  = & \frac{P_x}{2} \psi \psi_x - \frac{P_t}{2} \phi_x \psi + \frac{P}{2} \phi_x \bigg( \Big( \frac{R \xi}{v} \Big)_x - \Big( \frac{P}{v} \Big)_x \phi + R_1 \bigg) - \tfrac{R}{\gamma-1} \xi_x \psi_t \\[2mm]
  & + \kappa \Big( \frac{\xi_x}{v} - \frac{\Theta_x \phi}{v V} \Big)_x \psi_x - (p-P)( U_x + \psi_x ) \psi_x + R_2 \psi_x \, .
\end{aligned}
\end{eqnarray}
 Integrating (\ref{3.48}) on $[0, t] \times \mathbb{R}$ and using the inequality (\ref{3.46}), it holds that
\begin{eqnarray}\label{3.49}
\begin{aligned}
  & \int^t_0 \int_{\mathbb{R}} ( \phi^2_x + \psi^2_x ) dx d\tau \\[2mm]
  \lesssim & \| ( \phi, \psi, \xi ) \|^2_1 + \| ( \phi_0, \psi_0, \xi_0 ) \|^2_1
  + \int^t_0 \| \xi_{x}(\tau) \|^2_1 d\tau \\[2mm]
  & + \big( \tfrac{1}{4} + \delta + N(t) \big) \int^t_0 \int_{\mathbb{R}} ( \phi^2_x + \psi^2_x ) dx d\tau
  + \int^t_0 \int_{\mathbb{R}} ( | \Theta_{xx} | + | \Theta_x |)^2 ( \phi^2 + \psi^2 ) dx d\tau \\[2mm]
  & + \int^t_0 \int_{\mathbb{R}} | R_1 \phi_x + R_2 \psi_x | dx d\tau\, .
\end{aligned}
\end{eqnarray}
Similar as the estimates of (\ref{3.28}) and (\ref{3.29}), we can easily get
\begin{eqnarray}\label{3.50}
\begin{aligned}
  & \int^t_0 \int_{\mathbb{R}} { ( \phi^2_x + \psi^2_x ) } dx d\tau \\[2mm]
  \lesssim & \| ( \phi_0, \psi_0, \xi_0 ) \|^2_2 + \delta + ( \delta + N(t) ) \int^t_0 \| ( \phi_x, \psi_x ) (\tau) \|^2_1 d\tau \, .
\end{aligned}
\end{eqnarray}

Similarly, multiplying $(\ref{3.3})_{2x}$ by $-\frac{P}{2}\phi_{xx}$, $(\ref{3.3})_{3x}$ by $\psi_{xx}$, respectively, and integrating the result on $[0, t] \times \mathbb{R}$, then by using(\ref{3.46}), we can also get
\begin{eqnarray}\label{3.51}
\begin{aligned}
  & \int^t_0 \int_{\mathbb{R}} ( \phi^2_{xx} + \psi^2_{xx} ) dx d\tau \\[2mm]
  \lesssim & \| ( \phi_0, \psi_0, \xi_0 ) \|^2_2 + \delta + ( \delta + N(t) ) \int^t_0 \| ( \phi_x, \psi_x )(\tau) \|^2_1 d\tau \, .
\end{aligned}
\end{eqnarray}
Adding the results of (\ref{3.50}) with (\ref{3.51}), we can get (\ref{3.47}) and this  completes the proof of Lemma 3.5.
\end{proof}

Inserting  (\ref{3.47}) into (\ref{3.46}), then it yields (\ref{3.8}), that is the result of Proposition 3.2.\\

{\bf{ Acknowledgments}}: The authors are grateful to Professor A. Matsumura his support and advice. This work work was supported by the Fundamental Research Funds for the Central Universities and three grants from the National Natural Science Foundation of China under contracts 11301405, 11671309 and 11731008, respectively.

 \end{document}